\documentclass{amsart}

\usepackage{srcltx} 
\usepackage{amsmath,amssymb,amsthm,amsfonts}
\usepackage[utf8]{inputenc}

\newtheorem{Thm}{Theorem}
\newtheorem{Le}[Thm]{Lemma}
\newtheorem{Cor}[Thm]{Corollary}

\newtheorem{Rem}[Thm]{Remark}

\newcommand{\dist}{\operatorname{dist}}

\newcommand{\bD}{\mathbf{D}}

\newcommand{\oQ}{{Q_0}}
\newcommand{\rdN}{\eR^{d\times N}}
\newcommand{\rd}{\eR^{d}}
\newcommand{\camp}{{\mathcal{L}}}
\newcommand{\mean}[1]{\langle #1\rangle}
\newcommand{\dx}{}

\begin{document}

\title[$L^q$ estimates for nonlinear non-diagonal systems]{Gradient $L^q$ theory for a class of non-diagonal nonlinear elliptic systems}\thanks{M. Bul\'{\i}\v cek and V. M\'acha were supported by the Czech Science Foundation (grant no. 16-03230S). V. M\'acha was also supported by RVO: 67985840. M. Bul\'{\i}\v cek, P. Kaplick\'y and V. M\'acha acknowledge membership in the Ne\v cas Center of Mathematical Modeling (http://ncmm.karlin.mff.cuni.cz).}

\author[M.~Bul\'{\i}\v{c}ek]{Miroslav Bul\'{\i}\v{c}ek}
\address{Mathematical Institute, Faculty of Mathematics and Physics, Charles University, Sokolovsk\'{a} 83,
186 75 Praha 8, Czech Republic}
\email{mbul8060@karlin.mff.cuni.cz}

\author[M.~Kalousek]{Martin Kalousek}
\address{University of W\"urzburg, Institute of Mathematics, Emil-Fischer-Straße 40, 97074 W\"urzburg, Germany}
\email{kalousek@karlin.mff.cuni.cz}

\author[P.~Kaplick\'{y}]{Petr Kaplick\'y}
\address{Department of Mathematical Analysis, Faculty of Mathematics and Physics, Charles University, Sokolovsk\'{a} 83,
186 75 Praha 8, Czech Republic}
\email{kaplicky@karlin.mff.cuni.cz}

\author[V.~M\'{a}cha]{V\'aclav M\'acha}
\address{Institute of Mathematics of the Czech Academy of Sciences, \v Zitn\'a 25, 115 67 Praha 1, Czech Republic}
\email{macha@math.cas.cz}




\newcommand{\dashint}{-\!\!\!\!\!\!\int}
\newcommand{\qs}{q^\star}

\newcommand{\vv}{\mathcal V}
\newcommand{\eS}{\mathbb S}
\newcommand{\loc}{{loc}}
\newcommand{\eR}{\mathbf{R}}
\newcommand{\eN}{\mathbf{N}}
\newcommand{\nm}[1]{\left\|#1\right\|}
\newcommand{\dual}[1]{\left\langle#1\right\rangle}
\newcommand{\spt}{\operatorname{spt}}
\newcommand{\csubset}{\subset\subset}

\newcommand{\dd}{\mbox{d}}
\newcommand{\ltws}{local in time weak solution}
\setlength{\jot}{2.5ex}
\newcommand{\isigma}{s}
\newcommand{\sss}{s}

\newcommand{\Du}{\bD v}
\newcommand{\du}{\nabla v}
\newcommand{\xxl}{\mathcal{X}_{\ell}}
\newcommand{\wwl}{\mathcal{W}_{\ell}}
\newcommand{\bbl}{\mathcal{B}_{\ell}}
\newcommand{\atl}{\mathcal{A}_{\ell}}
\newcommand{\etl}{\mathcal{E}_{\ell}}
\newcommand{\etr}{\mathcal{E}}
\newcommand{\ddt}{\frac{d}{dt}}
\newcommand{\dt}{\partial_t}
\newcommand{\norm}[3]{\|{#1}\|_{#2}^{#3}}
\newcommand{\dalka}[3]{\operatorname{dist}_{#3}(#1,#2)}
\newcommand{\eend}[1]{\hfill \ensuremath{\Box}}
\newcommand{\beggin}[1]{Proof. }
\def\diff{\mathsf{d}}
\def\diver{\mathop{\mathrm{div}}\nolimits}
\def\diam{\mathrm{diam}}

\begin{abstract}
We consider a class of nonlinear non-diagonal elliptic systems with $p$-growth and establish the $L^q$-integrability for all $q\in [p,p+2]$ of any weak solution provided  the corresponding right hand side belongs to the corresponding Lebesgue space and the involved elliptic operator asymptotically satisfies the $p$-uniform ellipticity,   the so-called splitting condition and it is continuous with respect to the spatial variable. For operators satisfying the uniform $p$-ellipticity condition the higher integrability is known for $q\in[p,dp/(d-2)]$ and for operators having the so-called Uhlenbeck structure, the theory is valid for all $q\in [p,\infty)$. The key novelty of the paper is twofold. First, the statement uses only the information coming from the asymptotic operator and second, and more importantly, by using the splitting condition, we are able to extend the range of possible $q$'s significantly whenever $p<d-2$.
\end{abstract}

\maketitle

\section{Introduction}\label{intro}


This paper focuses on  the properties of a local weak solution $v\in W^{1,p}_{loc}(\Omega; \eR^N)$, $p>1$ to the  following nonlinear system
\begin{equation}
\begin{split}
-\diver \mathcal{A}(x,\nabla v(x)) &= -\diver G(x) \textrm{ in } \Omega.\label{eqn}
\end{split}
\end{equation}
Here $\Omega \subset \eR^d$ is an open set with $d\ge 2$, $\mathcal{A}:\Omega \times \eR^{d\times N} \to \eR^{d\times N}$ with $N\in \mathbb{N}$ is a Carath\'{e}odory mapping and $G\in L^{p'}_{loc}(\Omega; \eR^{d\times N})$ is arbitrary. In addition, we assume that the mapping $\mathcal{A}$ has the $(p-1)$ growth and satisfies the $p$-coercivity, i.e., there exist $\alpha_0>0$ and $\alpha_1\ge 0$ such that for almost all $x\in \Omega$ and all $\eta\in \eR^{d\times N}$ the following holds
\begin{align}
|\mathcal{A}(x,\eta)|&\le \alpha_1(1+|\eta|^{p-1}),\label{p-growth}\\
\mathcal{A}(x,\eta) \cdot \eta &\ge \alpha_0|\eta|^p - \alpha_1.\label{p-coercivity}
\end{align}
Our general aim is to establish the local $L^q$ estimates with $q>p$ for the gradient of any distributional solution $v\in W^{1,p}_{loc}(\Omega, \eR^N)$ to \eqref{eqn} provided that $G\in L^{\frac{q}{p-1}}(\Omega; \eR^{d\times N})$, i.e., we want to show that
\begin{equation}
\label{dream1}
\int_{Q} |\nabla v|^q \le C(R,q)\left(1+\int_{4Q} |G|^{\frac{q}{p-1}} + \left(\int_{4Q} |\nabla v|^p\right)^{\frac{q}{p}}\right)
\end{equation}
for every cube $Q\subset 4Q\subset \Omega$. Under the general assumptions \eqref{p-growth}--\eqref{p-coercivity} such a theory however cannot be built, see the counterexamples in \cite{Ne77,Ser64,Ser65,Ser65I,SvYa02}, without additional conditions on the smoothness of $\mathcal{A}$ with respect to the spatial variable and the dependence of $\mathcal{A}$ on $\eta$. The only available positive result in this generality, i.e., assuming only \eqref{p-growth}--\eqref{p-coercivity}, is based on the use of the reverse H\"{o}lder inequality, see e.g. \cite{giaquinta}, and we know that there exists $\varepsilon>0$ depending only on $\alpha_0, \alpha_1, d,p$ such that \eqref{dream1} holds for all $q\in [p,p+\varepsilon)$. On the other hand, due to the fundamental work \cite{Uhl77}, we have the full regularity (in this case $\mathcal{C}^{1,\alpha}$-regularity) of solution to \eqref{eqn} with  $\mathcal{A}(x,\eta) \sim |\eta|^{p-2}\eta$  and smooth right hand side. This information can be used for establishing the nonlinear Calder\'{o}n-Zygmund theory that have its roots in articles \cite{iwa2} and \cite{iwa1} and we know that the estimate \eqref{dream1} is true for all $q\in [p,\infty)$ for these operators. In addition, the result can be even strengthen to the BMO setting, see \cite{DieKapSch11}. Moreover, following the idea presented e.g. in \cite{Mar1996,MarPap2006}, we know that the only operator that drives the gradient theory is the asymptotic operator (note that if it exists, it is necessarily homogeneous of order $p-1$)
$$
\mathcal{A}_{\infty}(x,\eta):=\lim_{\lambda \to \infty} \frac{\mathcal{A}(x,\lambda \eta)}{\lambda^{p-1}},
$$ 
which will be also used in this paper. 

Our general goal is to find a more general structural assumptions on $\mathcal{A}$ or $\mathcal{A}_{\infty}$ respectively than those introduced by Uhlenbeck, that allow us to establish the estimate \eqref{dream1} for larger values of $q$ than those given by the reverse H\"{o}lder technique. The first step in this direction was already done.  The method presented in \cite{CaPe1998} somehow separates arguments from function theory and theory of partial differential equations but is, however, not directly applicable to PDE's for which we have relatively poor information about the solution, which is the case of non-diagonal systems, with right hand side and in the situation that only low regularity of comparison problem is available.  The approach to this situation was carefully developed in \cite[Section~7]{KriMin06} for operators being uniformly $p$-monotone, i.e., for operators satisfying for some $\delta_0\ge 0$ and all $\eta,\xi \in \eR^{d\times N}$ and almost all $x\in \Omega$ the following
\begin{equation}\label{p-mono}
\alpha_0 (\delta_0+ |\eta|^2)^{\frac{p-2}{2}} |\xi|^2 \le \frac{\partial \mathcal{A}(x,\eta)}{\partial \eta} \cdot (\xi \otimes \xi)\le \alpha_1 (\delta_0+ |\eta|^2)^{\frac{p-2}{2}} |\xi|^2,
\end{equation}
see also  \cite{DieKa2013} for even more general operators related to the generalized Sokes system.\\
Under such assumption, one can use the higher regularity technique to deduce that for ``smooth" $G$ the solution $v$ to \eqref{eqn} belongs to $W^{1,dp/(d-2)}_{loc}(\Omega; \eR^{d\times N})$ and incorporating this information with the fundamental paper \cite{CaPe1998}, one can conclude that \eqref{dream1} is valid for all $q\in [p,dp/(d-2)]$. 
This approach is used also in this paper. 

Recently, an everywhere $\mathcal{C}^{0,\alpha}$-regularity theory for elliptic systems was built in \cite{BuFre2012} (and further extended in \cite{BuFreStei2014,BuFreStei2015}) for operators $\mathcal{A}$ being of potential form and satisfying the so-called splitting condition but not requiring the monotonicity assumption \eqref{p-mono}. Thus, the aim of the paper is to combine the everywhere $\mathcal{C}^{0,\alpha}$ theory with the $p$-monotonicity assumption in order to get the bigger class of possible $q$'s for which \eqref{dream1} holds. It will be seen in the paper that a new borderline for these operators then will be $q\in[p,p+2]$, which is a significant improvement in case that $p<d-2$, i.e., in high dimensions and for low value of $p$.

We finish the introduction by precise formulation of the assumption on $\mathcal{A}$ and by the statement of our key result.
\subsection{Assumptions}
In this subsection we state the assumptions on $\mathcal{A}$. Beside \eqref{p-growth}--\eqref{p-coercivity} we shall assume some qualitative properties of the asymptotic operator $\mathcal{A}_{\infty}$. To do it, we use  a Carath\'{e}odory function $F:\Omega\times \eR^{d\times N}\to \eR$ and we denote
$$
F_{\eta}(x,\eta):=\frac{\partial F(x,\eta)}{\partial \eta} :\Omega \times \eR^{d\times N} \to \eR^{d\times N},
$$
which is supposed to be Carath\'{e}odory. Moreover, we assume that $F_{\eta}$ is the asymptotic operator to $\mathcal{A}$, i.e., for all $\eta$ and almost all $x\in \Omega$ there holds
\begin{equation}
\label{ident-as}
F_{\eta}(x,\eta)=\lim_{\lambda \to \infty} \frac{\mathcal{A}(x,\lambda\eta)}{\lambda^{(p-1)}}
\end{equation}
and we assume the certain uniformity of the above limit. More precisely, we require that  for all $\varepsilon >0$ there exists $\lambda>0$ such that for all $\zeta \in \eR^{d\times N}$ fulfilling $|\eta|\ge \lambda$ and almost all $x\in \Omega$ the inequality
\begin{equation}
|F_{\eta}(x,\eta)-\mathcal{A}(x,\eta)|\le \varepsilon |\eta|^{p-1}.\label{asse}
\end{equation}
holds. 
Note that from the assumption \eqref{p-growth}--\eqref{p-coercivity} and \eqref{ident-as}, it directly follows that the potential $F$ is homogeneous of degree $p$, i.e., $F(x,\lambda\eta)=\lambda^pF(x,\eta)$ and also $F_{\eta}$ is homogeneous of degree $(p-1)$, i.e., $F_{\eta}(x,\lambda\eta)=\lambda^{p-1}F_{\eta}(x,\eta)$. Moreover, assuming \eqref{p-growth}--\eqref{p-coercivity} and \eqref{ident-as}--\eqref{asse}, we deduce that 
\begin{align}
\alpha_0 |\eta|^p\le F_\eta(x,\eta) \cdot \eta &=pF(x,\eta)\le \alpha_1|\eta|^p, \label{5.55}\\
|F_\eta(x,\eta)| &\le \alpha_1|\eta|^{p-1} \label{5.56}
\end{align}
for all $\eta\in \eR^{d\times N}$ and almost all $x\in \Omega$.

Finally, we specify the requirements for $F$. 
The theory 
requires also certain smoothness with respect to the spatial variable and therefore we assume that for all $\varepsilon>0$ there exists $\delta>0$ such that for all $x_1,x_2\in \Omega$ satisfying $|x_1-x_2|\le \delta$ and for all $\eta \in \eR^{d\times N}$ we have  
\begin{align}\label{5.75}
\left|F_\eta(x_1,\eta) - F_\eta(x_2,\eta)\right|\leq \varepsilon |\eta|^{p-1}.
\end{align}
Concerning the $p$-monotonicity, we transfer it to $F$ and assume that  for almost all $x\in\Omega$, the function  $F(x,\cdot)\in C^2(\rdN\setminus\{0\})$ and that $\partial^2_{\eta^2}F(x,\eta)$ is measurable for all $\eta\in \eR^{d\times N}$ and further we require that for all $\eta,\xi \in\rdN$ with $\eta \neq 0$ and a.a. $x\in \Omega$
\begin{align}
\alpha_0|\eta|^{p-2}|\xi|^2&\le \frac{\partial^2 F(x,\eta)}{\partial \eta^2} \cdot (\xi \otimes \xi)\le \alpha_1|\eta|^{p-2}|\xi|^2.\label{5}
\end{align}
Finally, following \cite{BuFre2012}, we shall assume the splitting\footnote{Splitting conditions refers to the fact that $A^{\alpha \beta}$ does not depend on $i,j$. Recall, the indexes $i,j$ correspond to derivatives w.r.t. $x_i$, $x_j$ respectively.} condition for $F_{\eta}$.  It means we assume that there are  symmetric matrix-valued Carath\'{e}odory function $A: \Omega\times \eR^{d\times N}\to \eR^{N\times N}$ and a symmetric matrix valued measurable $b: \Omega \to  \eR^{d\times d}$  such that
\begin{align}
F_{\eta_{i}^\alpha}(x,\eta)&=\sum_{\beta =1}^N \sum_{j=1}^d A^{\alpha \beta}(x,\eta) b_{ij}(x)\eta_j^\beta \label{6}
\end{align}
for all $\eta \in \eR^{d\times N}$,  all $i=1,\ldots, d$, all $\alpha=1,\ldots, N$ and a.a. $x\in \Omega$. Observe that \eqref{6} is an additional structural condition which is independent of the fact, that the $F_{\eta^{\alpha}_i}$ come from a potential. Moreover,  we require  that for all $\mu \in \eR^{N}$, all $\eta \in \eR^{d\times N}$ fulfilling $\eta\neq 0$, all $v\in \eR^d$ and almost all $x\in \Omega$ the inequalities
\begin{equation}
\begin{split}
\alpha_0 |\eta|^{p-2}|\mu|^2&\le\sum_{\alpha,\beta=1}^N A^{\alpha \beta}(x,\eta)\mu^\alpha \mu^{\beta} \le \alpha_1 |\eta|^{p-2}|\mu|^2,\\
\alpha_0 |v|^2 &\le \sum_{i,j=1}^d b_{ij}(x) v_i v_j \le \alpha_1 |v|^2\label{6.1}
\end{split}
\end{equation}
hold.

\subsection{Main result}
The main result is summarized in the following theorem.
\begin{Thm}\label{mainclaim}
Let $\Omega \subset \eR^d$ be an open set, $p\in(1,\infty)$, $\mathcal{A}$ satisfy \eqref{p-growth}--\eqref{p-coercivity} and \eqref{ident-as} and the corresponding $F$ fulfill  \eqref{5}--\eqref{6.1}. Then there exists $q_0>p+2$, $R_0>0$ and $c\ge 0$ depending only on $d,N,\alpha_0,\alpha_1,p$ and $F$ such that for any $G\in L^{q/(p-1)}_{loc}(\Omega;\eR^{d\times N})$ with  $q\in (p,q_0)$ and $v\in W^{1,p}_{loc}(\Omega;\eR^N)$ solving \eqref{eqn} in the sense of distribution, the following estimate
\begin{equation}\label{finalest}
\dashint_{Q} |\nabla v|^q\leq
c\left[1+\dashint_{4Q}|G|^{q/(p-1)} +\left(\dashint_{4Q} |\nabla v|^p\right)^{\frac{q}{p}}\right]
\end{equation}
holds true for all cubes $Q\subset \eR^d$ with sidelength smaller than $R_0$ and satisfying $4Q\subset \Omega$.
\end{Thm}


\begin{Rem}\label{Rq0}
The precise value of $q_0$ is equal to $p+ 2/(1-\alpha)$ where $\alpha$ is a coefficient of H\"older continuity which appear in Theorem \ref{Thm:HoldCont}.
\end{Rem}

The remaining part of this paper is devoted to the proof of Theorem \ref{mainclaim}. In the next section we introduce some observations concerning the Young functions and Nikolski\u{\i} spaces. An approximate problem as well as estimates of its solution are presented in Section \ref{comparison}. The proof of Theorem \ref{mainclaim} is concluded in Section \ref{proofofmain} where we use Lemma \ref{CZCP} in order to compare a solution to \eqref{eqn} with the solution of the approximate problem.

\section{Auxiliary tools and notation}

Throughout this paper we work with a Young function $\varphi(t)=\frac{t^p}p$ which clearly satisfies \cite[Assumption~1]{DE}. We denote by $\varphi^*$ its complementary function (see cf. \cite{DE}).
In order to simplify estimates we introduce the function $V(A)=|A|^{(p-2)/2}A$ for $A\in \eR^{d\times N}$. With this notation we get from \eqref{5.55} and \eqref{5.56}
\begin{equation}\label{in:2}
\exists C>0,\forall \eta\in\eR^{d\times N},x\in\Omega: |V(\eta)|^2\leq CF_\eta(x,\eta)\cdot\eta, |F_\eta(x,\eta)|\leq C\varphi'(|\eta|).
\end{equation}

There exists $c,\ c',\ c''> 0$ such that
\begin{equation}\label{in:4}
\varphi^*(\varphi'(|\eta|))\leq c\varphi(|\eta|) \leq c' |V(\eta)|^2\leq c'' \varphi^*(\varphi'(|\eta|))
\end{equation}
see \cite[(2.3)]{DE}. Further, \cite[Lemma~3]{DE} gives
\begin{multline}\label{est:hammer}
\exists c,C>0,\forall A,B\in\rdN:\\
c|V(A)-V(B)|^2\leq \varphi''(|A|+|B|)|A-B|^2\leq C|V(A)-V(B)|^2.
\end{multline}
%
Moreover, from \eqref{5} it follows that there exists $c>0$ such that for all $\eta,\xi\in\rdN$, $x\in\Omega$
\begin{equation}\label{in:5}
(F_\eta(x,\eta) - F_\eta(x,\xi))\cdot(\eta - \xi) \geq c|V(\eta) - V(\xi)|^2,
\end{equation}
compare with \cite[proof of Lemma~21]{DE}.

In the following part of the paper, estimates in Nikolski\u{\i} spaces $N^{\lambda,p}$ appear. The definition and basic properties of these spaces can be found in \cite[Chapter 4]{Triebel} where they are denoted $B^\lambda_{p,\infty}(\Omega)$. Here we call them $N^{\lambda,p}(\Omega)$. We recall their definition for reader's convenience. Several equivalent definitions can be also found in \cite[Theorem 4.2.2/2]{Triebel}. Let $Q\subset\eR^d$ be a cube with sidelength $R>0$. Let $p\in[1,\infty]$,
$\lambda\in(0,1)$. For $h\in\rd$, $h\neq 0$ we denote $Q_h=\{x\in Q;x+h\in Q\}$. We define the space $N^{\lambda,p}$ as
\def\difh{\diff_h}
\def\difs{\diff_s}
\begin{equation*}
	N^{\lambda,p}(Q)=\{w\in L^{p}(Q);[w]_{N^{\lambda,p}(Q)}<\infty\},
\end{equation*}
where the seminorm $[\cdot]_{N^{\lambda,p}}$ is defined as
\begin{equation*}
	[w]_{N^{\lambda,p}(Q)}=
R^\lambda R^{-d/p}\sup_{h\neq 0,h\in B_{R}}\frac{\|\difh w\|_{L^p(Q_{h})}}{|h|^\lambda}.
\end{equation*}
with
$$\difh f = f(x+h) - f(x).$$
The appropriate norm is $\|\cdot\|_{N^{\lambda,p}}:= R^{-d/p}\|\cdot\|_{L^p} + [\cdot]_{N^{\lambda,p}}$.

If $\lambda\in(1,2)$ there are two possibilities how to define seminorm $[w]_{N^{\lambda,p}(Q)}$
\begin{align*}
[w]_{N^{\lambda,p}(Q)}^{(1)}&=R^\lambda R^{-d/p}\sup_{h\neq 0,h\in B_{R}}\frac{\|\difh^2 w\|_{L^p(Q_{2h})}}{|h|^\lambda}\\
[w]_{N^{\lambda,p}(Q)}^{(2)}&=R^\lambda R^{-d/p}\sup_{h\neq 0,h\in B_{R}}\frac{\|\difh \nabla w\|_{L^p(Q_{h})}}{|h|^{\lambda-1}}.
\end{align*}
The norms $\|\cdot\|_{N^{\lambda,p}(Q)}^{(1)}:=R^{-d/p}\|\cdot\|_{L^p(Q)}+[\cdot]_{N^{\lambda,p}(Q)}^{(1)}$ and $\|\cdot\|_{N^{\lambda,p}(Q)}^{(2)}:=R^{-d/p}\|\cdot\|_{L^p(Q)}+[\cdot]_{N^{\lambda,p}(Q)}^{(2)}$
are equivalent according to \cite[Theorem 4.4.2/2]{Triebel}. It is important that the appearing constants does not depend on $R>0$.

We need the following property of Nikolski\u{\i} spaces
\begin{equation}\label{ReduProp}
	w\in N^{\lambda,p}(Q)\Leftrightarrow\nabla w\in N^{\lambda-1,p}(Q),
\end{equation}
which is a consequence of \cite[Theorem~4.4.2/2]{Triebel}.

\section{Comparison}\label{comparison}
\subsection{Local $L^q$ regularity result}
The origin of the technique used to prove the theorem can be traced back to \cite{iwa2,iwa1}. It is based on a local comparison of a given solution with a solution to a suitably chosen approximate problem. The method was further developed and clearly described in \cite{CaPe1998}. In this article the authors deal with the interior $L^q$ theory in the situation that for the approximate problem an $W^{1,\infty}$ local estimates are available. A suggestion what to do in the case that one can obtain for the approximate problem interior $W^{1,q_0}$ estimates, $q_0<+\infty$ only, appears also there. This situation is investigated (among other regularity results) in \cite{KriMin06}.

For a cube $Q$ and $\alpha>0$ we define a cube $\alpha Q$ as a cube with the same center as $Q$ whose edges are parallel and have length $\alpha$ times length of edges of $Q$. Furthermore, for a dyadic cube $Q_k$ we denote its predecessor by $\tilde Q_k$. This notation is effective throughout the paper.
The proof of the main theorem is based on the following lemma.

\begin{Le}\label{CZCP}
Let $1\leq r<s<t<\infty$, $Q\subset \mathcal O$ be a cube and $Q_k$ be dyadic cubes obtained from $Q$. Further, let $f\in L^{s}(4Q)$. $g\in L^s(4Q)$ and $w\in L^r(4Q)^n$. Then there exists $\varepsilon_0>0$ independent of $Q$ such that the following implication holds:

If there exists $\varepsilon\in(0,\varepsilon_0)$ such that for every dyadic cube $Q_k\subset Q$, $Q_k\neq Q$ there exists $w_a\in L^p(4\tilde Q_k)^n$ with following properties:
\begin{align}
\label{CZCP1}
\left(\dashint_{2\tilde Q_k} |w_a|^t\dx\right)^{\frac{1}{t}} &\leq \frac{C}{2} \left(\dashint_{4\tilde Q_k}|w_a|^r\dx\right)^{\frac{1}{r}},\\
\label{CZCP2}
\dashint_{4\tilde Q_k} |w_a|^r\dx &\leq C\dashint_{4\tilde Q_k} |w|^r\dx + C\dashint_{4\tilde Q_k} |g|^r\\
\label{CZCP3}
\dashint_{4\tilde Q_k} |w-w_a|^r\dx&\leq \varepsilon \dashint_{4\tilde Q_k} |w|^r\dx + C \dashint_{4\tilde Q_k} |f|^r\dx
\end{align}
then $w\in L^s(Q)^n$. Positive constants $C$ and $\varepsilon$ are independent on $Q_k$, $w_a$ and $w$.\\
Furthermore, there exists a positive constant $c$ independent of $f$, $g$ and $w$ such that
\begin{equation}\label{cpcon}
\left(\dashint_{Q} |w|^s\right)^{\frac1s}\leq
c\left[\left(\dashint_{4Q}|f|^{s}\right)^{\frac1s}+\left(\dashint_{4Q}|g|^s\right)^{\frac1s} +\left(\dashint_{4Q} |w|^r\right)^{\frac1r}\right].
\end{equation}
\end{Le}
The proof relays on considerations presented in \cite{CaPe1998} in the situation $f=0$. In this work, the lemma is also closely connected to systems of partial differential equations.  The same technique is used in \cite{KriMin06} to show $L^q$ theory for more general systems compared to systems considered here with $f\neq 0$ up to the boundary. Since that problem is more general, authors have worse estimates for the comparison problem which results in smaller range of $q$'s for which the theory holds. In \cite{DieKa2013} the method was used to generalized Stokes problem with $f\neq 0$. The presented version is just a slight change of \cite[Lemma 2.7]{MaTi}. Since its proof is a modification of the one in \cite{CaPe1998} and \cite{MaTi} we skip it in this article.

\subsection{Comparison problem and its estimates}

\subsubsection{The comparison problem}
To apply Lemma~\ref{CZCP} we need to construct for a given solution of \eqref{eqn} a function $w_a$ with properties \eqref{CZCP1},\eqref{CZCP2} and  \eqref{CZCP3}. When studying the elliptic systems of equations such a function is usually found as a solution to a suitable approximate system. It is the same in our situation.

Let us assume that we have a local solution $v\in W^{1,p}(\Omega; \eR^N)$ of the problem \eqref{eqn}. 
Hereinafter, $\oQ\subset\Omega$ will be an arbitrary but fixed cube with the centre $x_0$, such that $Q_0\subset B_1(x_0)$.

We consider a comparison problem
\begin{equation}
\begin{split}
\diver F_\eta(x_0,\nabla u)  &= 0 \textrm{ in } \oQ,\\
u&=v \textrm{ on } \partial \oQ \label{sys}
\end{split}
\end{equation}
Existence and uniqueness of a weak solution to this problem follows by theory of monotone operators due to Assumption~\eqref{5}. It was applied for example in \cite[Lemma~3.1]{DieKa2013}  to a similar problem motivated by fluid mechanics, with an additional constraint $\diver u=0$ in $\oQ$.

We formulate the result in the next lemma.
\begin{Le} \label{lem:ex+un}
Let  $v\in W^{1,p}(\Omega)$ be a local weak solution to \eqref{eqn}. Then for every $\kappa>0$ there exists $R>0$ such that for every $\oQ$ with $\diam(\oQ)<R$ there is a unique weak solution $u\in W^{1,p}(\oQ)$
  to the problem \eqref{sys} satisfying
  \begin{gather}
    \label{lem:diff0}\exists C>0:\int_{\oQ}|V(\nabla u)|^2\leq C\int_{\oQ}|V(\nabla v)|^2,\\
    \label{lem:diff}\int_{\oQ}|V(\nabla u)-V(\nabla v)|^2\leq C(\kappa)\left(1 + \int_{\oQ} |G|^{\frac p{p-1}}\right)+\kappa\int_{Q_0}|V(\nabla v)|^2.
  \end{gather}
The constants $C,C(\kappa)$ are independent of $u$, $v$, $\oQ$, $x_0$.
\end{Le}
\begin{proof}
We skip the proof of existence and uniqueness since it follows line by line \cite[Lemma~3.1]{DieKa2013}.

Also the proofs of estimates \eqref{lem:diff0} and \eqref{lem:diff} are very similar to ones in \cite{DieKa2013} but since the elliptic operator $\mathcal A$ in \eqref{eqn} depends also on $x$ we present them here. Let us start by testing \eqref{sys} by $u-v$ to get
\begin{multline*}
\int_{\oQ} |V(\nabla u)|^2 \leq c\int_{\oQ} F_\eta (x_0,\nabla u) \nabla u = c\int_{\oQ} F_\eta(x_0,\nabla u) \nabla v \\
\leq \delta\int_{\oQ} \varphi^*(|F_\eta(x_0,\nabla u)|) + C_\delta\int_{\oQ} \varphi(|\nabla v|).
\end{multline*}
From \eqref{in:2} and \eqref{in:4} we get
$$
\int_{\oQ} |V(\nabla u)|^2 \leq \delta\int_{\oQ} |V(\nabla u)|^2 + C_\delta\int_{\oQ} |V(\nabla v)|^2,
$$
i.e. \eqref{lem:diff0}.

We subtract the weak formulation of \eqref{sys} from \eqref{eqn} and test the result by $u-v$ extended by zero outside $\oQ$ in order to get
\begin{multline}\label{druhyodhad}
\int_{\oQ} (F_\eta(x_0,\nabla u) - \mathcal A(x,\nabla v))(\nabla u - \nabla v) = \int_{\oQ} G(\nabla u - \nabla v)\\
\leq C_\kappa\int_{\oQ} \varphi^*(|G|) + \kappa \int_{\oQ} \varphi(|\nabla u - \nabla v|)\leq C_\kappa\int_{\oQ} \varphi^*(|G|) + C\kappa \int_{\oQ} |V(\nabla v)|^2,
\end{multline}
where we used \eqref{lem:diff0}.

Due to \eqref{in:5} we have
\begin{equation}\label{czcpjedna}
\int_{\oQ}(F_\eta(x_0,\nabla u) - F_\eta(x_0,\nabla v))(\nabla u-\nabla v)\geq C\int_{\oQ}|V(\nabla u) - V(\nabla v)|^2
\end{equation}

The smoothness property of $F$ (see \eqref{5.75}) together with the Young inequality and \eqref{lem:diff0} yield
\begin{multline}\label{czcpdva}
\int_{\oQ} (F_\eta(x_0,\nabla v) - F_\eta(x,\nabla v))(\nabla u-\nabla v) \leq \kappa \int_{\oQ} \varphi'(|\nabla v|)|\nabla u - \nabla v| \\ \leq \kappa C\int_{\oQ} |V(\nabla v)|^2
\end{multline}
provided $\diam(\oQ)$ is sufficiently small.

Finally, we deduce from \eqref{p-growth}, \eqref{ident-as}, \eqref{asse}, \eqref{5.55}, \eqref{lem:diff0} and the Young inequality that
\begin{multline}\label{czcptri}
\int_{\oQ}(F_\eta(x,\nabla v) - \mathcal A(x,\nabla v))(\nabla u - \nabla v)\\ \leq \int_{\oQ\cap\{|\nabla v|\leq \lambda(\kappa)\}} |F_\eta(x,\nabla v) - \mathcal A(x,\nabla v)||\nabla u - \nabla v|\\ + \int_{\oQ\cap |\nabla v|>\lambda(\kappa)} |F_\eta(x,\nabla v) - \mathcal A(x,\nabla v)| |\nabla u - \nabla v|\\ \leq C(\kappa) \int_\oQ |\nabla u - \nabla v| + \kappa \int_\oQ \varphi'(|\nabla v|)|\nabla u - \nabla v|\leq C(\kappa) + \kappa \int_\oQ |V(\nabla v)|^2. 
\end{multline}
where we use also the fact that $\oQ$ is bounded.
Since $\varphi^*(t)=(p-1) t^{\frac p{p-1}}/p$ the claim follows from \eqref{druhyodhad}, \eqref{czcpjedna}, \eqref{czcpdva} and \eqref{czcptri}.
\end{proof}

\subsubsection{Known estimates for approximative solution}

\begin{Le}\label{Lem:IntReg}
Let $u$ be the weak solution of \eqref{sys}.Then $V(\nabla u)\in W^{1,2}_{loc}(\oQ)$. Moreover we have for any cube $Q\subset 2Q\subset B$
\begin{equation}\label{CaccIneq}
	\int_{Q}|\nabla V(\nabla u)|^2\leq \frac{c}{R^2}\int_{4Q}|V(\nabla u)|^2,
\end{equation}
where $R$ is the length of the edge of $Q$.
\end{Le}

\begin{proof}
We follow the proof of \cite[Theorem~11]{DE} that deals with a more complicated situation where the elliptic term depends also on $x$.

Set $\theta\in C^\infty_C(2Q)$, $\chi_{Q}\!\leq\!\theta\!\leq\!\chi_{2Q}$, $|\nabla\theta|\leq \frac{c}{R}$. We choose $h,s\in\eR^d\setminus\{0\}$, $|s|\leq|h|\leq R$. We apply $\difs$ to \eqref{sys} and obtain for any $w\in W^{1,p}_0(2Q)$

	\begin{equation}\label{DifferenceEqn}
		\int_{2Q} (F_\eta(x_0,\nabla u(x+s)-F_\eta(x_0,\nabla u(x))\cdot\nabla w(x)\dx=0.
	\end{equation}
Let us denote $A_s(x):=F_\eta(x_0,\nabla u(x+s))-F_\eta(x_0,\nabla u(x))$. As in the proof of \cite[Theorem 4]{DE}, we choose some $\hat q>1$ and $c>0$ such that

\begin{equation}\label{InThetaPow}
	\varphi_a(\theta^{\hat q-1}t)\leq c\theta^{\hat q}\varphi_a(t)
\end{equation}
holds uniformly in $a,t\geq 0$.
We use in \eqref{DifferenceEqn} the test function $w:=\theta^{\hat q}\difs u$ to get

\begin{equation}\label{DefI1andI2}
I_1:=\int_{2Q} A_s(x)\theta^{\hat q}\cdot \nabla \difs u(x)\dx=-\int_{2Q} A_s(x){\hat q}\theta^{\hat q-1}\cdot\difs u(x)\otimes\nabla\theta\dx=:I_2.
\end{equation}

We obtain
\begin{equation*}
	I_1\geq c\int_{2Q}\theta^{\hat q}\left|\difs V(\nabla u)\right|^2
\end{equation*}

by \eqref{in:5} and in the same manner as in the proof of \cite[Lemma 12]{DE}
\begin{equation*}
	I_2\leq \varepsilon \int_{2Q}\theta^{\hat q}\frac{|s|}{|h|}\dashint_0^{|s|}|\diff_\lambda V(\nabla u)|^2\dd\lambda\dd x+c_{\varepsilon}\frac{|h|^2}{R^2}\int_{2Q}|V(\nabla u)|^2.
\end{equation*}

We integrate the resulting inequality over $s\in [0,h)$ (c.f. \cite[Lemma 12]{DE}) to get
\begin{multline*}
	\dashint_0^{|h|}\dashint_{Q}\left|\diff_\lambda V(\nabla u)\right|^2\dd x \dd \lambda\leq \varepsilon \dashint_0^{|h|}\dashint_{2Q}\eta^{\hat q}\frac{|s|}{|h|}\dashint_0^{|s|}|\diff_\lambda V(\nabla u)|^2\dd\lambda\dd x \dd s\\
	+c_{\varepsilon}\frac{|h|^2}{R^2}\dashint_{2Q}|V(\nabla u)|^2.
\end{multline*}

Then we apply \cite[(4.20)]{DE} and \cite[Lemma 13]{DE} to obtain
\begin{equation*}
	\dashint_{Q}\left|\difh V(\nabla u)\right|^2\leq c\frac{|h|^2}{R^2}\int_{4Q}|V(\nabla u)|^2.
\end{equation*}
We divide this inequality by $|h|^2$ and use the characterization of Sobolev spaces via difference quotients to conclude \eqref{CaccIneq}.

A simple covering argument gives the statement.
\end{proof}

We use further local estimates to $u$. In \cite[Theorem 1.1]{BuFre2012} the following theorem appears.

\begin{Thm}\label{Thm:HoldCont}
\label{T1}
Let $p\in (1,\infty)$ and $F$ satisfy \eqref{5}--\eqref{6.1}. Then there exists $\alpha>0$ such that the unique weak solution $u$ to \eqref{sys} belongs to $\mathcal{C}^{\alpha}_{loc}(\oQ;\eR^N)$. Moreover, for all $Q \subset \oQ$ we have the estimate
$$
\|u\|_{\mathcal{C}^\alpha(Q)}\le C(Q)\|u\|_{W^{1,p}(\oQ)}.
$$
In addition, there exists $C>0$ such that for all $x_0\in \oQ$ and all $\varepsilon>0$ the solution $u$ satisfies the following 
potential inequality
\begin{equation}
\begin{split}
\int_{B_r(x_0)} &\frac{\varepsilon |\nabla u|^p}{r^\varepsilon|x-x_0|^{d-p-\varepsilon}} + \frac{|\nabla u|^{p-2}|\nabla u \cdot (x-x_0)|^2}{|x-x_0|^{d-p+2}}\; \dx \\
&\qquad \le C\left(\frac{r}{R}\right)^{\alpha p}\left(R^p + \int_{B_R(x_0)}\frac{|\nabla u|^p}{R^{d-p}}\; \dx\right)
\end{split}\label{Campan}
\end{equation}
 for all $0<2r<R<\textrm{ dist }(x_0,\partial \oQ).$
\end{Thm}

\begin{Cor}
Assume, in addition to the assumptions of Theorem \ref{Thm:HoldCont}, that $p\in(1,d)$. Then there is $c>0$ such that for any $Q\subset 3Q\subset \oQ$
	\begin{equation}\label{HSEst}
		[u]_{\mathcal{C}^\alpha(Q)}\leq cR^{-\alpha+1}\left(\left(\dashint_{2Q}|\nabla u|^p\right)^\frac{1}{p}+1\right),
	\end{equation}
	where $R$ is the length of the edge of $Q$.
\end{Cor}
	\begin{proof}
Let $B\subset 2B\subset \oQ$ be a ball with radius $\rho$. From \eqref{Campan} it follows by Poincar\'e inequality that
$$
[u]^p_{\camp^{p,d+p\alpha}(B)}\leq C\rho^{p-\alpha p}\dashint_{2B} (|\nabla u|^p+1).
$$

Let now $Q\subset 6\sqrt dQ\subset Q_0$ be a cube with sidelength $R$. Previous inequality together with the integral characterization of H\"older continuous functions,
(see Lemma~\ref{lem:hchcf} in appendix) gives
\begin{equation*}
[u]_{\mathcal{C}^\alpha(Q)}
\leq [u]_{\mathcal{C}^\alpha(\sqrt d B)}
\leq C[u]_{\camp^{p,d+p\alpha}(3\sqrt d B)}
\leq CR^{1-\alpha}\left(\left(\dashint_{6\sqrt d B} |\nabla u|^p\right)^{\frac 1p}+1\right)
\end{equation*}
The statement of the corollary then follows by a suitable covering argument. 
\end{proof}

%

\subsubsection{Interpolation of the regularity results}

The integrability of $\nabla u$ can be improved locally to $p+2$ as the next lemma states.

\begin{Le}\label{lem:8}

Under the assumption of Theorem \ref{T1} there is $c>0$ such that for any $Q\subset 2Q\subset \oQ$
\begin{equation}\label{P+2Int}
	\left(\dashint_{Q}|\nabla u|^{p+2}\right)^\frac{1}{p+2}\leq c\left(\left(\dashint_{2Q}|\nabla u|^p\right)^\frac{1}{p} + 1\right)
\end{equation}
\end{Le}
\begin{proof}
	
	Let $\theta\in C^\infty_C(\oQ)$ be such that $\chi_{Q}\leq\theta\leq\chi_{2Q}$, $|\nabla \theta|\leq{c}/{R}$, where $R$ is the length of the edge of $Q$ as usual.

We denote
	\begin{align*}
			&\mathsf{A(r)}=\left(\dashint_{rQ}\theta^{p+2}|\nabla u|^{p+2}\right)^\frac{1}{p+2}\\
			&\mathsf{B(r)}=\|u-(u)_{rQ}\|_{L^\infty(rQ)}\\
			&\mathsf{C(r)}=\left(\dashint_{rQ}|\nabla u|^p\right)^\frac{1}{p}
	\end{align*}
	
	As a consequence of \eqref{CaccIneq} and \eqref{HSEst} we obtain
	\begin{gather}\label{LIEst}
		\mathsf{B}(r)\leq \sup_{x\in rQ}\dashint_{Q}|u(x)-u(y)|\dd y\leq R^\alpha[u]_{\mathcal{C}^\alpha(Q)}\leq c R\mathsf{C}(3r) + cR\quad\mbox{for $r\in[1,10]$},\\
\dashint_{2Q}|\nabla u|^{p-2}|\nabla^2 u|^2\leq CR^{-2}C^p(4).\label{CaccIneq1}
\end{gather}
	
	
	

	In order to estimate the quantity $\mathsf A$, we use integration by parts, H\"{o}lder's and Young's inequalities, \eqref{LIEst}, and \eqref{CaccIneq1}
	
	\begin{equation}\label{IIEst}
		\begin{split}
		  \left(\mathsf{A}(2)\right)^{p+2} &=\dashint_{2Q}\theta^{p+2}|\nabla u|^2|\nabla u|^p=\dashint_{2Q}\theta^{p+2}\nabla u\cdot \nabla(u-(u)_{2Q})|\nabla u|^p\\
			&=-(p+2)\dashint_{2Q}\theta^{p+1}|\nabla u|^p\nabla\theta\otimes(u-(u)_{2Q})\cdot\nabla u\\
			&-\dashint_{2Q}\theta^{p+2}|\nabla u|^p(u-(u)_{2Q})\cdot\Delta u\\
			&-p\dashint_{2Q}\theta^{p+2}|\nabla u|^{p-2}[\nabla u](u-(u)_{2Q})\cdot[\nabla^2u]\nabla u\\
			&\leq \frac{c}{R}\mathsf{B}(2)\dashint_{2Q}\theta^{p+1}|\nabla u|^{p+1}+c\mathsf{B}(2)\dashint_{2Q}\theta^{p+2}|\nabla u|^p|\nabla^2 u|\\
			&\leq \frac{c}{R}\mathsf{B}(2)\left(\mathsf{A}(2)\right)^{p+1}\\&+c\mathsf{B}(2)\dashint_{2Q}\left(\theta^{p+2}|\nabla u|^{p-2}|\nabla^2 u|^2\right)^\frac{1}{2}\left(\theta^{p+2}|\nabla u|^{p+2}\right)^\frac{1}{2}\\
			&\leq c(\varepsilon) R^{-p-2}\left(\mathsf{B}(2)\right)^{p+2}+c(\varepsilon)\left(\mathsf{B}(2)\right)^2\dashint_{2Q}|\nabla u|^{p-2}|\nabla^2 u|^2\\
			&+2\varepsilon\left(\mathsf{A}(2)\right)^{p+2}\leq c\left(\mathsf{C}(6)\right)^{p+2}+2\varepsilon\left(\mathsf{A}(2)\right)^{p+2} + c.
		\end{split}
	\end{equation}
	The demanded statement follows by a simple covering argument.
\end{proof}

In \eqref{P+2Int} we proved a reverse H\"older inequality which could be improved by Gehring's theorem. Here we proceed in a different way and trace the dependence of integrability of $\nabla u$ on the constant $\alpha$ from Theorem~\ref{Thm:HoldCont}.

Lemma~\ref{lem:8} does not show the full strength of the information from Theorem~\ref{Thm:HoldCont}. We improve it in the next lemma which holds for $p\geq 2$ only.

\begin{Le}\label{lem:pgreater2}
Let $p\geq 2$. Under the assumptions of Theorem \ref{T1} set $q\in(p,p+2/(1-\alpha))$. Then $\nabla u\in L^q_{loc}(\oQ)$ and there is $C>0$ such that for any $Q\subset 2Q\subset \oQ$
	\begin{equation}\label{est:pgreater2}
\left(\dashint_Q|\nabla u|^q\right)^{\frac1q}\leq c \left(\left(\dashint_{2Q}|\nabla u|^p\right)^{\frac1p} + 1\right).
	\end{equation}
The constant $C$ may depend on $q$.
	\end{Le}
\begin{Rem} Note that the function $f(\alpha)=(2+p(1-\alpha))/(1-\alpha)$ is increasing and continuous on $[0,1)$, $f(0)=p+2$ and $f(1-)=+\infty$. So the information in Lemma~\ref{lem:pgreater2} seems to be satisfactory.
\end{Rem}

	\begin{proof}
	Lemma \ref{Lem:IntReg} together with \eqref{est:hammer} and the fact that $p\geq 2$ implies $\nabla u\in N^{\frac{2}{p},p}_{loc}(\oQ)$ and according to \eqref{ReduProp} $u\in N^{1+\frac{2}{p},p}_{loc}(\oQ)$. Moreover, \eqref{CaccIneq} implies
	\begin{equation}\label{GrNikEst1}
\begin{aligned}
		R^{-(p+2)}[\nabla u]^p_{N^{\frac{2}{p},p}(Q)}&\leq
	\dashint_{2Q}|\nabla V(\nabla u)|^2\leq \frac{c}{R^2}\dashint_{4Q}|V(\nabla u)|^2\\
&\leq
\frac{c}{R^2}\dashint_{4Q}|\nabla u|^p,
\end{aligned}
\end{equation}
where $R$ is a sidelength of cube $Q$.
We compute
\begin{equation}\label{est:nik1}
\begin{aligned}
\left([u]^{(1)}_{N^{1+\frac{2}{p},p}(Q)}\right)^p\leq
\left(\nm{u-\mean{u}_{Q}}_{N^{1+\frac{2}{p},p}(Q)}^{(1)}\right)^p\leq
\left(\nm{u-\mean{u}_{Q}}_{N^{1+\frac{2}{p},p}(Q)}^{(2)}\right)^p\\=
\left(\left(\dashint_Q|u-\mean{u}_{Q}|^p\right)^{\frac1p}+[\nabla u]_{N^{\frac{2}{p},p}(Q)}\right)^{p}
\leq R^p\dashint_{4Q}|\nabla u|^p.
\end{aligned}
\end{equation}
We choose $h\in\eR^d$. Then we obtain by \eqref{HSEst} and \eqref{est:nik1}
\begin{align*}
\dashint_{Q}|\difh^2 u|^q &
\leq \|\difh^2 u\|_{L^\infty(Q)}^{q-p}\dashint_{Q}|\difh^2 u|^{p}\\
&\leq c|h|^{\alpha(q-p)}[u]_{\mathcal{C}^\alpha(5Q)}^{q-p}\left(\frac{|h|}{R}\right)^{p+2}\left([u]^{(1)}_{N^{1+\frac2p,p}(5Q)}\right)^p\\
&\leq c |h|^{p+2+\alpha(q-p)} cR^{(q-p)(1-\alpha)-2}\left(\left(\dashint_{20Q}|\nabla u|^p\right)^{\frac{q-p}{p}+1} + \dashint_{20Q}|\nabla u|^p\right).
	\end{align*}
Define $\mu(q)=(2+p+\alpha(q-p))/q$. Notice that $\mu:[p,+\infty)\to(1,2)$. We get
$$
[u-\mean u_Q]_{N^{\mu(q),q}(Q)}^{(1)}\leq cR\left(\left(\dashint_{20Q}|\nabla u|^p\right)^{\frac{1}{p}} + 1\right).
$$
Assume for a while that $q<p^*:=pd/(d-p)$. The Poincar\'e inequality implies
\begin{equation}\label{est:quaterfinal}
\nm{u-\mean u_Q}_{N^{\mu(q),q}(Q)}^{(2)}\leq c\nm{u-\mean u_Q}_{N^{\mu(q),q}(Q)}^{(1)}\leq cR\left(\left(\dashint_{20Q}|\nabla u|^p\right)^{\frac{1}{p}} + 1\right).
\end{equation}
As
\begin{multline*}
\dashint_Q |\nabla u - \mean{\nabla u}|^q\leq cR^{(\mu(q) - 1)q} \dashint_{2Q}\dashint_{2Q} \frac{|\nabla u(x+h) - \nabla u(x)|^q}{|h^{(\mu(q) - 1)q}|}{\rm d}x{\rm d}h\\ \leq c\left(\frac 1R [u -\mean{u}]^{(2)}_{N^{\mu(q),q}(4Q)}\right)^q,
\end{multline*}
the H\"older inequality and \eqref{est:quaterfinal} yield
\begin{equation}\label{est:semifinal}
\left(\dashint_{Q}|\nabla u|^q\right)^{\frac{1}{q}}
\leq c\left(\left(\dashint_{20Q}|\nabla u|^p\right)^{\frac{1}{p}} + 1\right).
\end{equation}
The lemma is proved under additional assumption $q<p^*$ and with a larger integration domain on the right hand side. We use \eqref{est:semifinal} in \eqref{est:quaterfinal} in order to increase the validity of the lemma for $q\in[p,\min(p^{**},(2+p(1-\alpha))/(1-\alpha))$. The statement of the lemma is obtained after finite number of iterations and an application of a simple covering argument.
\end{proof}

Now we prove the same result as in Lemma~\ref{lem:pgreater2} also in the case $p<2$. It will need some preparations. Unfortunately, we do not know the method that would work in both cases.

\begin{Le}\label{lem:second-derivatives}
Let $Q\subset \rd$ be a cube with sidelength $R>0$,  $1<p\leq r<2$ and $w\in W^{2,p}(5Q)$. Then  for all $h\in B_R(0)$
\begin{equation}\label{est:second-derivatives}
\begin{aligned}
|h|^{-2r}&\dashint_Q|\difh^2 w|^r\leq \dashint_{5Q}|\nabla^2 w|^r\\
&\leq\left(\dashint_{5Q}|\nabla w|^{p-2}|\nabla^2 w|^2\right)^\frac{r}{2}\left(\dashint_{5Q}|\nabla w|^{\frac{r(2-p)}{2-r}}\right)^\frac{2-r}{2}.
\end{aligned}
\end{equation}
\end{Le}
\begin{proof}
The inequality $|h|^{-2r}\dashint_Q|\difh^2 w|^r\leq \dashint_{5Q}|\nabla^2 w|^r$ is easy consequence of Theorem of Newton and Leibniz, Jensen's inequality, Fubini Theorem and a proper change of variables. Further we estimate by H\"older's inequality
\begin{equation*}\label{est:sd2}
\begin{aligned}
\dashint_{5Q}|\nabla^2 w|^r&=\dashint_{5Q}\big(|\nabla w|^{p-2}|\nabla^2 w|^2\big)^{\frac r2}|\nabla w|^{r\frac{(2-p)}2}\\
			&\leq\left(\dashint_{5Q}|\nabla w|^{p-2}|\nabla^2 w|^2\right)^{\frac r2}\left(\dashint_{5Q}|\nabla w|^{\frac{r(2-p)}{(2-r)}}\right)^{\frac{2-r}{2}}.
\end{aligned}
\end{equation*}

\end{proof}

\begin{Le}\label{lem:pless2}
Let $p<2$. Under the assumptions of Theorem \ref{T1} set $q\in(p,p+2/(1-\alpha))$. Then $\nabla u\in L^q_{loc}(\oQ)$ and there is $C>0$ such that for any $Q\subset 2Q\subset \oQ$
	\begin{equation}\label{est:pless2}
\left(\dashint_Q|\nabla u|^q\right)^{\frac1q}\leq C \left(\left(\dashint_{2Q}|\nabla u|^p\right)^{\frac1p} + 1\right).
	\end{equation}
The constant $C$ may depend on $q$.
	\end{Le}
\begin{proof}
We will proceed by iterations. We assume that \eqref{est:pless2} holds for some $q$ (this is always true for $q=p$) and show that if $q<p+2/(1-\alpha)$ then \eqref{est:pless2} holds also with $q$ replaced by $\gamma\in(q,2q(2-\alpha)/(2-p+q)/(1-\alpha))$. Since $2q(2-\alpha)/(2-p+q)/(1-\alpha)-q$ is a positive and continuous function of $q$ on $[p,p+2/(1-\alpha))$ the statement \eqref{est:pless2} follows.

Let \eqref{est:pless2} hold. We show that this estimate continues to hold if we replace $q$ with $\gamma\in(q,2q(2-\alpha)/(2-p+q)/(1-\alpha))$. Note that this interval is nonempty. Set $r=2q/(2-p+q)$. We emphasize that if $q\in[p,2)$ then $r$ belongs to $[p,+\infty)$ and $r\leq q$ for $q\geq p$. The coefficient $r$ satisfies $r(2-p)/(2-r)=q$.

Lemma~\ref{lem:second-derivatives} applied to $u$ gives together with Lemma~\ref{Lem:IntReg}
\begin{equation}\label{est:pl1}
\begin{aligned}
\dashint_Q|\difh^2 u|^r
&\leq c|h|^{2r}\left(\dashint_{5Q}|\nabla u|^{p-2}|\nabla^2 u|^2\right)^\frac{r}{2}\left(\dashint_{5Q}|\nabla u|^q\right)^{\frac{r(2-p)}{2q}}\\
&\leq c|h|^{2r}\left(R^{-2}\dashint_{10Q}|\nabla u|^p\right)^{\frac1p\frac{rp}{2}}\left(\dashint_{5Q}|\nabla u|^q\right)^{\frac1q\frac{r(2-p)}{2}}\\
&\leq c|h|^{2r}R^{-r}\left(\dashint_{10Q}|\nabla u|^q\right)^{\frac rq}.
\end{aligned}
\end{equation}

We proceed similarly as in the proof of Lemma~\ref{lem:pgreater2}. Successive application of \eqref{est:pl1} and \eqref{HSEst} yields
\begin{align*}
\dashint_{Q}|\difh^2 u|^\gamma &
\leq \|\difh^2 u\|_{L^\infty(Q)}^{\gamma-r}\dashint_{Q}|\difh^2 u|^{r}\\
&\leq c|h|^{\alpha(\gamma-r)}[u]_{\mathcal{C}^\alpha(5Q)}^{\gamma-r}|h|^{2r}R^{-r}\left(\dashint_{10Q}|\nabla u|^q\right)^{\frac rq}\\
&\leq cR^\gamma\left(\frac{|h|}R\right)^{2r+\alpha(\gamma-r)}\left(\left(\dashint_{10Q}|\nabla u|^q\right)^{\frac \gamma q} + \left(\dashint_{10Q} |\nabla u|^q\right)^{\frac rq}\right).
	\end{align*}
For $\mu:=2r+\alpha(\gamma-r)\in(\gamma,2\gamma)$ we obtain
$$
[u]_{N^{\mu/\gamma,\gamma}(Q)}\leq cR\left(\left(\dashint_{100Q}|\nabla u|^q\right)^{\frac 1 q} + 1\right).
$$
This allows to conclude
$$
\left(\dashint_{Q}|\nabla u|^\gamma\right)^{\frac 1 \gamma}\leq c\left(\left(\dashint_{10Q}|\nabla u|^q\right)^{\frac 1 q} + 1\right)\leq c\left( \left(\dashint_{20Q}|\nabla u|^p\right)^{\frac 1 p} + 1\right),
$$
similarly as at the end of Lemma~\ref{lem:pgreater2}. The statement of the lemma follows by covering argument.
\end{proof}

Now we may proceed to the proof of the main result.

\section{The proof of Theorem \ref{mainclaim}}\label{proofofmain}
The statement of the theorem follows from Lemma~\ref{CZCP} applied to $w=|\nabla v|^{p/2}$ and $w_a = |\nabla u|^{p/2} + 1$ where $v$ is a solution to \eqref{eqn} and $v$ is a solution to \eqref{sys} on some cube $Q_0\subset\Omega$. The inequality \eqref{CZCP2} with $g\equiv 1$ follows from  \eqref{lem:diff0} just by considering the definition of $V$. Similarly we also get \eqref{CZCP1} from \eqref{est:pgreater2} or \eqref{est:pless2}. 

By the triangle inequality there exists $C>0$ such that
\begin{equation}\label{est:diffV}
\forall A,B\in\rdN: ||A|^{\frac p2}-|B|^{\frac p2}|^2\leq C|V(A)-V(B)|^2.
\end{equation}
Thus $|w -w_a|^2\leq C(|V(\nabla u) - V(\nabla v)|^2 + 1)$ and the desired estimate \eqref{CZCP3} follows from \eqref{lem:diff} provided $\kappa<\epsilon_0$, $\epsilon_0$ being from Lemma~\ref{CZCP}. At this moment we obtain a restriction on the sidelength of cube $Q_0$. As all assumptions of Lemma~\ref{CZCP} are met we get
\begin{equation}\label{est:final}
\left(\dashint_{Q} |\nabla v|^s\right)^{\frac 1{s}}\leq
c\left[1+\left(\dashint_{4Q}|G|^{s/(p-1)}\right)^{\frac 1{s}} +\left(\dashint_{4Q} |\nabla v|^p\right)^{\frac 1{p}}\right]
\end{equation}
for any $s\in (p,p+2/(1-\alpha))$.

\section{Appendix}

Here we recall \cite[Theorem III.1.2]{giaquinta} and its proof. We focus on the fact that the constant depends only on dimension $d$.
\begin{Le}\label{lem:hchcf}
Let $B_0\subset\rd$ be a ball, $u\in L^1(3B_0)$. Define
\begin{gather*}
[u]_{\camp^{p,d+p\alpha}(3B_0)}:=
\sup\{(\dashint_B|u-\mean{u}_B|^p)^{\frac1p}\rho^{-\alpha}; B\subset 3B_0, \mbox{$\rho$ is radius of $B$}\},\\
[u]_{C^\alpha(B_0)}:=
\sup\{|u(x)-u(y)||x-y|^{-\alpha}; x,y\in B_0\}.
\end{gather*}
Then there is $C>0$ that may depend only on $d$ such that
$$[u]_{C^\alpha(B_0)}\leq C[u]_{\camp^{p,d+p\alpha}(3B_0)}.$$
\end{Le}
\begin{proof}
First we realize that $u\in C^\alpha_\loc(B_0)$ by \cite[Theorem III.1.2]{giaquinta}.
Fix $x,y\in B_0$ and denote $R=\dist(x,y)$, $B_{k}^x=B_{2^{-k}R}(x)$, $ B_{k}^y=B_{2^{-k}R}(y)$. It holds $\lim_{k\to+\infty}\mean{u}_{B^x_k}=u(x)$, $\lim_{k\to+\infty}\mean{u}_{B^x_k}=u(x)$ and $$|\mean{u}_{B^x_k}-\mean{u}_{B^x_{k+1}}|\leq 2^d\left(\dashint |u-\mean{u}_{B^x_{k}}|^p\right)^{\frac1p}\leq 2^dR^\alpha2^{-k\alpha}[u]_{\camp^{p,d+p\alpha}(3B_0)}.$$
 Consequently, $$|\mean{u}_{B^x_0}-u(x)|\leq 2^dR^\alpha\frac{2^\alpha}{2^\alpha-1}[u]_{\camp^{p,d+p\alpha}(3B_0)}$$ and similarly at point $y$.

Further, let $B_{-1}^{xy}$ be a ball with radius $2R$ containing $B_0^x$ and $B_0^y$. We get
\begin{multline*}|\mean{u}_{B^x_0}-\mean{u}_{B^y_0}|\leq |\mean{u}_{B^x_0} - \mean{u}_{B_{-1}^{xy}}| + |\mean{u}_{B^y_0} - \mean{u}_{B_{-1}^{xy}}|\\ \leq C\dashint_{B_{-1}^{xy}}|u - \mean{u}|\leq CR^\alpha [u]_{\camp^{p,d+p\alpha}}(3B_0).\end{multline*}
Altogether, 
$$|u(x)-u(y)|\leq CR^\alpha[u]_{\camp^{p,d+p\alpha}(3B_0)}.$$
\end{proof}


\bibliographystyle{amsplain}
\bibliography{literatura}

\providecommand{\bysame}{\leavevmode\hbox to3em{\hrulefill}\thinspace}
\providecommand{\MR}{\relax\ifhmode\unskip\space\fi MR }
\providecommand{\MRhref}[2]{%
  \href{http://www.ams.org/mathscinet-getitem?mr=#1}{#2}
}
\providecommand{\href}[2]{#2}
\begin{thebibliography}{10}

\bibitem{BuFre2012}
M.~Bul{\'{\i}}{\v{c}}ek and J.~Frehse, \emph{{$C^\alpha$}-regularity for a
  class of non-diagonal elliptic systems with {$p$}-growth}, Calc. Var. Partial
  Differential Equations \textbf{43} (2012), no.~3-4, 441--462. \MR{2875647}

\bibitem{BuFreStei2014}
M.~Bul{\'{\i}}{\v{c}}ek, J.~Frehse, and M.~Steinhauer, \emph{Everywhere
  $\mathcal c^\alpha$-estimates for a class of nonlinear elliptic systems with
  critical growth}, Adv. Calc. Var. \textbf{7} (2014), no.~2, 139--204.
  \MR{3187915}

\bibitem{BuFreStei2015}
\bysame, \emph{On h\"older continuity of solutions for a class of nonlinear
  elliptic systems with $p-$growth via weighted integral techniques}, Ann. Mat.
  Pura Appl. (4) \textbf{194} (2015), no.~4, 1025--1069. \MR{3357693}

\bibitem{CaPe1998}
L.~A. Caffarelli and I.~Peral, \emph{On {$W^{1,p}$} estimates for elliptic
  equations in divergence form}, Comm. Pure Appl. Math. \textbf{51} (1998),
  no.~1, 1--21. \MR{1486629 (99c:35053)}

\bibitem{DE}
L.~Diening and F.~Ettwein, \emph{Fractional estimates for non-differentiable
  elliptic systems with general growth}, Forum Math. \textbf{20} (2008), no.~3,
  523--556. \MR{2418205 (2009h:35101)}

\bibitem{DieKa2013}
L.~Diening and P.~Kaplick{\'y}, \emph{{$L^q$} theory for a generalized {S}tokes
  system}, Manuscripta Math. \textbf{141} (2013), no.~1-2, 333--361.
  \MR{3042692}

\bibitem{DieKapSch11}
L.~Diening, P.~Kaplick{\'y}, and S.~Schwarzacher, \emph{Bmo estimates for the
  p-laplacian}, Nonlinear Anal. \textbf{75} (2012), no.~2, 637--650.

\bibitem{giaquinta}
M.~Giaquinta, \emph{Multiple integrals in the calculus of variations and
  nonlinear elliptic systems}, Annals of Mathematics Studies, vol. 105,
  Princeton University Press, Princeton, NJ, 1983. \MR{717034 (86b:49003)}

\bibitem{iwa2}
T.~Iwaniec, \emph{On {$L^p$}-integrability in pdes and quasiregular mappings
  for large exponents}, Ann. Acad. Sci. Fenn. Ser. A I Math. \textbf{7} (1982),
  no.~2, 301--322. \MR{6866647 (84h:30026)}

\bibitem{iwa1}
\bysame, \emph{Projections onto gradient fields and $l^p$-estimates for
  degenerated elliptic operators}, Studia Math. \textbf{75} (1983), no.~3,
  293--312. \MR{722254 (85i:46037)}

\bibitem{KriMin06}
J.~Kristensen and G.~Mingione, \emph{The singular set of minima of integral
  functionals}, Arch. Ration. Mech. Anal. \textbf{180} (2006), no.~3, 331--398.
  \MR{2214961 (2007d:49059)}

\bibitem{MaTi}
V.~M\'acha and J.~Tich\'y, \emph{Higher integrability of generalized stokes
  system under perfect slip boundary conditions}, J. Math. Fluid Mech
  \textbf{16} (2014), 832--845.

\bibitem{Mar1996}
P.~Marcellini, \emph{Everywhere regularity for a class of elliptic systems
  without growth conditions}, Ann. Scuola Norm. Sup. Pisa Cl. Sci. (4)
  \textbf{23} (1996), no.~1, 1--25. \MR{1401415}

\bibitem{MarPap2006}
P.~Marcellini and G.~Papi, \emph{Nonlinear elliptic systems with general
  growth}, J. Differential Equations \textbf{221} (2006), no.~2, 412--443.
  \MR{2196484}

\bibitem{Ne77}
J.~Ne\v{c}as, \emph{Example of an irregular solution for a class of elliptic
  systems without growth conditions}, Theory of nonlinear operators (Proc.
  Fourth Internat. Summer School, Acad. Sci., Berlin, 1975), Akademie-Verlag,
  Berlin, 1977. \MR{0509483}

\bibitem{Ser64}
J.~Serrin, \emph{Pathological solutions of elliptic differential equations},
  Ann. Scuola Norm. Sup. Pisa (3) \textbf{18} (1964), 385--387.

\bibitem{Ser65}
\bysame, \emph{Isolated singularities of solutions of quasi-linear equations},
  Acta Math. \textbf{113} (1965), 219--240.

\bibitem{Ser65I}
\bysame, \emph{Singularities of solutions of nonlinear equations}, Proc.
  Sympos. Appl. Math., vol. XVII, Amer. Math. Soc., Providence, R.I., 1965.

\bibitem{Triebel}
H.~Triebel, \emph{Interpolation theory, function spaces, differential
  operators}, North-Holland Mathematical Library, vol.~18, North-Holland
  Publishing Co., Amsterdam-New York, 1978. \MR{503903 (80i:46032b)}

\bibitem{Uhl77}
K.~Uhlenbeck, \emph{Regularity for a class of non-linear elliptic systems},
  Acta Math. \textbf{138} (1977), no.~3-4, 219--240.

\bibitem{SvYa02}
V.~\v{S}ver\'ak and X.~Yan, \emph{Non-lipschitz minimizers of smooth uniformly
  convex functionals}, Proc. Natl. Acad. Sci. USA \textbf{99} (2002), no.~24,
  15269--15276. \MR{1946762 (2003h:49066)}

\end{thebibliography}

\end{document}